\documentclass[11pt]{amsart}

\usepackage{amsfonts,amsmath,latexsym,amssymb,verbatim,amsbsy,amsthm}
\usepackage{graphicx}
\usepackage{caption}
\usepackage{float}
\usepackage{wrapfig}

\usepackage[top=1in, bottom=1in, left=1in, right=1in]{geometry}

\usepackage[dvipsnames]{xcolor}

\usepackage[colorlinks=true, pdfstartview=FitV, linkcolor=RoyalBlue,citecolor=ForestGreen, urlcolor=blue]{hyperref}

\theoremstyle{plain}
\newtheorem{THEOREM}{Theorem}[section]

\newtheorem{theorem}[THEOREM]{Theorem}

\newtheorem{lemma}[THEOREM]{Lemma}

\theoremstyle{definition}

\newtheorem{definition}[THEOREM]{Definition}

\theoremstyle{remark}

\newtheorem{remark}[THEOREM]{Remark}

\newcommand{\thm}[1]{Theorem~\ref{#1}}
\newcommand{\lem}[1]{Lemma~\ref{#1}}

\newcommand{\sect}[1]{Section~\ref{#1}}

\DeclareMathOperator{\Supp}{Supp} %
\DeclareMathOperator{\Tr}{Tr} %
\def \a {\alpha}
\def \b {\beta}

\def \d {\delta}

\def \e {\varepsilon}

\def \l {\lambda}
\def \n {\nabla}

\def \D {\Delta}

\def \bd {{\bf d}}

\def \br {{\bf r}}

\def \bu {{\bf u}}
\def \bv {{\bf v}}
\def \bw {{\bf w}}
\def \bx {{\bf x}}

\def \cA {\mathcal{A}}

\def \cC {\mathcal{C}}
\def \cD {\mathcal{D}}

\def \cM {\mathcal{M}}
\def \cN {\mathcal{N}}
\def \cO {\mathcal{O}}

\def \cT {\mathcal{T}}

\def \cV {\mathcal{V}}

\def \cX {\mathcal{X}}

\def \dH {\dot{H}}

\newcommand{\R}{\ensuremath{\mathbb{R}}}   
\newcommand{\T}{\ensuremath{\mathbb{T}}}   

\renewcommand{\S}{\ensuremath{\mathbb{S}}} 

\def \Lip {\mathrm{Lip}}

\def \lan {\langle}
\def \ran {\rangle}

\def \p {\partial}

\def \db  {\, \mbox{d}\beta}
\def \dx  {\, \mbox{d}x}

\def \dt  {\, \mbox{d}t}

\def \dr  {\, \mbox{d}r}
\def \ds  {\, \mbox{d}s}

\def \ddt  {\frac{\mbox{d\,\,}}{\mbox{d}t}}

\begin{document}

\title[Unidirectional flocks]{ Existence and stability of unidirectional flocks in hydrodynamic Euler Alignment systems}

\author{Daniel Lear} \author{Roman Shvydkoy}

\address{Department of Mathematics, Statistics, and Computer Science, University of Illinois, Chicago}

\email{lear@uic.edu}
\email{shvydkoy@uic.edu}

\date{\today}



\thanks{\textbf{Acknowledgment.} Research of RS is supported in part by NSF
	grant DMS-1813351}

\maketitle

\begin{abstract}
In this note we reveal new classes of solutions to hydrodynamic Euler alignment systems governing collective behavior of flocks.  The solutions describe unidirectional parallel motion of agents, and are globally well-posed in multi-dimensional settings subject to a threshold condition similar to the one dimensional case. We develop the flocking and stability theory of these solutions and show long time convergence to traveling wave with rapidly aligned velocity field. 

In the context of multi-scale models introduced in \cite{ST-multi} our solutions can be superimposed into Mikado formations -- clusters of unidirectional flocks pointing in various directions. Such formations exhibit multiscale alignment phenomena and resemble realistic behavior of interacting large flocks.

\end{abstract}


\section{Introduction and statement of main results}
We consider the following  hydrodynamic Euler Alignment system for density $\rho(x,t)$ and velocity $\bu(x,t)=(u^{1}(x,t),\ldots,u^{n}(x,t))$  :
\begin{equation}\label{e:CSHydro}
(x,t)\in\mathbb{R}^{n}\times\mathbb{R}^{+}\qquad \left\{
\begin{split}
\partial_t \rho +\nabla\cdot (\rho \bu )&= 0, \\
\partial_t  \bu +\bu \cdot\nabla \bu &= \phi \ast (\rho  \bu ) - \bu (\phi \ast \rho),
\end{split}\right.
\end{equation}
subject to initial condition
$$\left( \rho(\cdot,t), \bu (\cdot,t)\right)|_{t=0}=(\rho_{0},\bu_{0}).$$
The crucial feature of the alignment term in \eqref{e:CSHydro} is its commutator representation given by
\begin{equation}\label{e:alignmentterm}
\mathcal{C}_{\phi}(\bu,\rho):= \phi \ast (\rho  \bu ) - \bu (\phi \ast \rho)=\int_{\R^n}\phi(x-y)\left(\bu(y)-\bu(x)\right)\rho(y)\,dy.
\end{equation}
Here, $\phi(x,y) = \phi(x-y)$ represents a positive communication kernel, which we assume is smooth and bounded throughout $\R^n$.

The system (\ref{e:CSHydro}) arises as a macroscopic realization of the Cucker-Smale agent-based dynamics \cite{CS2007a,CS2007b}, which describes collective motion of $N$ agents adjusting their directions to a weighted average of velocities of its neighbors:
\begin{equation}\label{e:CSagent}
(\bx_{i},\mathbf{v}_{i})\in\mathbb{R}^{n}\times\mathbb{R}^{n}\qquad \left\{
\begin{array}{@{} l l @{}}
\dot{\bx}_{i}&\hspace{-0.2 cm}= \mathbf{v}_{i}, \\
\dot{\mathbf{v}}_{i}&\hspace{-0.2 cm}=\frac{1}{N}\sum_{j=1}^{N}\phi(|\bx_{i}-\bx_{j}|)(\mathbf{v}_{j}-\mathbf{v}_{i}).
\end{array}\right. 
\end{equation}

We refer to \cite{HL2009,FK2017,HT2008} for full details and  rigorous derivations. Typical assumptions on $\phi(r)$ include monotonic decay at infinity and non-degeneracy, $\phi(r) >0$,  thus reflecting the intuition that alignment becomes weaker, yet persistent, as the distance becomes larger.  When communication remains sufficiently strong at infinity, expressed by the ``fat tail" condition
\begin{equation}\label{e:ft}
	\int_0^\infty \phi(r) \dr = \infty,
\end{equation}
the system \eqref{e:CSHydro} (as well as its discrete counterpart) exhibits alignment dynamics, that is for any global strong solution,  
\begin{equation}\label{alignment}
\mathcal{A}(t):=\underset{\left\lbrace x, y \right\rbrace\in\Supp \rho(\cdot,t)}{\text{max}}|\bu(x,t)-\bu( y ,t)|\to 0\qquad \text{as} \quad t\to\infty
\end{equation}
exponentially fast, and the diameter of the flock remains globally bounded:
$$
\mathcal{D}(t)\leq \bar{\cD}<\infty \qquad \text{where }\quad \mathcal{D}(t):=\underset{\left\lbrace x, y \right\rbrace\in\Supp \rho(\cdot,t)}{\text{max}}|x- y |.
$$
The tendency of solutions to flock made the system suitable for various technological and behavioral applications, see \cite{VZ2012,MT2014,Darwin}, and ignited a line of recent mathematical research, see Tadmor et al \cite{MT2014,ST-topo,ST1,TT2014}, Carrillo et al \cite{CCP2017,CCTT2016,CCMP2017} and references therein.

\begin{wrapfigure}{l}{0.38\textwidth}
	\centering
     \includegraphics[width=0.36\textwidth]{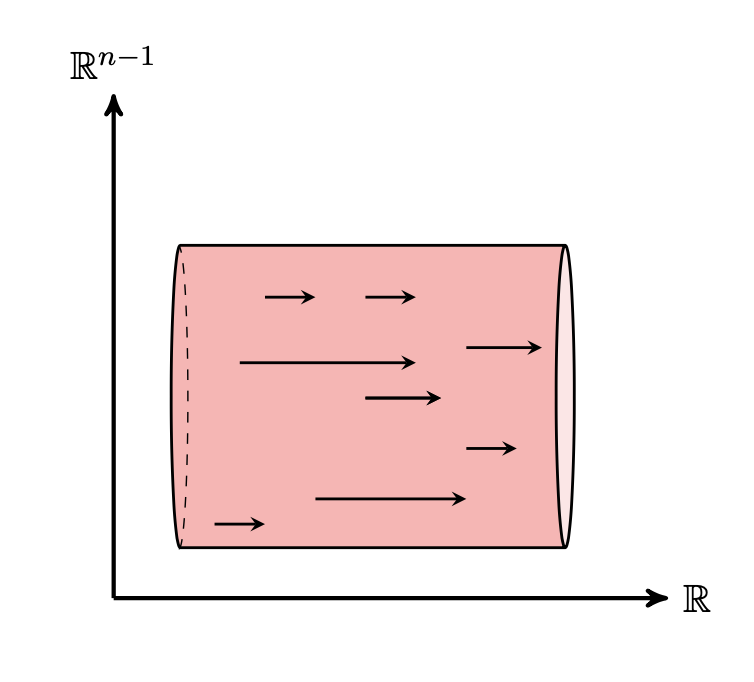}
\caption{Unidirectional flow}
\label{f:uni}
\end{wrapfigure}

Concerning well-posedness of the system \eqref{e:CSHydro}, we clearly have two competing mechanisms: Burger's transport and alignment regularization. The alignment regularization depends on how dense the flock is in that particular region of space. So, it is natural to expect that if $\rho$ has a vacuum region the solution blows up, see Tan \cite{Tan2017}.  The one-dimensional  theory \cite{CCTT2016} provided precise characterization for  global existence in terms of the threshold condition $\partial_x u_0(x)\geq -\left(\phi\ast\rho_0 \right)(x)$ for all $x\in\R$. An attempt to extend the study to dimension 2  was made in \cite{TT2014}, and improved in \cite{HeT2017} to a smallness condition on the spectral gap of the strain tensor $\n \bu_0 + \n^\perp \bu_0$. The well posedness in dimension 2 and higher, however, remains wide open.

In this paper we study a new class of flocks which do not fall under any of the previously considered categories -- these are unidirectional flows in multi-D setting, i.e.
\begin{equation}\label{e:ansatz}
\bu(x ,t):=\langle u(x,t),0,\ldots,0\rangle \qquad \text{for} \quad u:\mathbb{R}^n\times \mathbb{R}^{+}\rightarrow \mathbb{R},
\end{equation}
see Figure \ref{f:uni}. In view of the maximum principle the orientation of ansatz \eqref{e:ansatz} is preserved in time.  Note that the non-trivial component $u(x,t)$ may depend on all coordinates. So, our solutions exhibits features of a 1D flow, yet being on $\R^n$ represent solutions of a multi-D  system of scalar conservation laws:
\begin{equation}\label{e:CSHansatz}
(x,t)\in\mathbb{R}^n\times\mathbb{R}^{+}\qquad \left\{
\begin{array}{@{} l l @{}}
\partial_t \rho  +\partial_1(\rho u)&\hspace{-0.3 cm}=0, \\
\partial_t u+ \frac12\partial_1 (u^2)&\hspace{-0.3 cm}=\phi \ast (\rho u) - u (\phi\ast \rho).
\end{array}\right. 
\end{equation}
At the core of regularity theory of these solutions  is the analysis of the entropy quantity
\begin{equation}\label{entropy}
e:=\p_1 u+\phi\ast \rho,
\end{equation}
which happens to retain the same conservation law as in 1D:
\begin{equation}\label{e:elaw}
	\p_t e + \p_1 ( u e) = 0.
\end{equation}
See \cite{CCTT2016} for analysis in 1D, and \cite{LS-entropy} for the interpretation of $e$  as a topological entropy of the limiting flock.  The main global result is the following.

\begin{theorem}\label{t:gwp}
	Consider the multi-dimensional system (\ref{e:CSHansatz}) with smooth, monotone and positive communication kernel $\phi$. Suppose  $m \geq k+1 > \tfrac{n}{2}+2$ and $(u_0,\rho_0) \in H^m \times ( L^1_+ \cap W^{k,\infty})$. 
	\begin{itemize}
		\item {\em [Subcritical region]}. If $e_0(x) \geq 0$ for all $ x\in \R^n$, then  there exists a unique global solution 
			\begin{equation}\label{e:localclass2}
		(u,\rho) \in C_w([0,\infty); H^m \times (L^1_+ \cap W^{k,\infty})).
		\end{equation}
		Moreover if the kernel has fat tail \eqref{e:ft}, the following strong exponential flocking occurs:
		\[
		\cA(t) + |\nabla u(t)|_{L^{\infty}(\Supp\, \rho(\cdot,t))}+ |\nabla^2 u(t)|_{L^{\infty}(\Supp\, \rho(\cdot,t))}\lesssim e^{-\delta t}.
		\]
	and the density converges to a traveling wave solution: there exists $\bar{\rho} \in W^{1,\infty}$ such that 
		\[
		|\rho(t)-\bar{\rho}(\cdot - \bar{u} t)|_{C^{\gamma}}\lesssim e^{-\delta t} \qquad ( \forall\ 0<\gamma<1).
		\]
		\item {\em [Supercritical region]}. If $e_0(x_0)<0$ at some point $x_0\in \R^n$, then the solution blows up in finite time along the characteristics emanating from  $x_0$. 
	\end{itemize}
\end{theorem}
 The limiting velocity $\bar{u}$ is determined from the initial conditions due to conservation of momentum and mass
\[
\bar{u} := \frac{1}{\mathcal{M}} \int_{\R^n}\left(\rho\, u\right) (x,t)\dx, \quad  \mathcal{M}:=\int_{\R^n}\rho(x,t)\dx.
\]

The strong flocking was first proved in 1D under more stringent assumptions on the kernel in \cite{ST2}, and for 1D multi-scale models in \cite{ST-multi}. There does not seem to be a rule in either 1D or our situation  on how to determine the limiting density distribution of the flock  $\bar{\rho}$ -- this appears to be an \emph{emerging} quantity of the dynamics. However, the entropy estimates done in \cite{LS-entropy} show that, at least on the periodic domain the size of $e$ directly controls how far $\bar{\rho}$ is from the uniform distribution.

Our next result concerns perturbations of the constructed oriented solutions and their stability.

\begin{theorem}\label{t:stability}
Consider the multi-dimensional Euler Alignment system \eqref{e:CSHydro} on the periodic domain.   Let $(\bu_0,\rho_0) \in H^{m}\times (L_{+}^{1}\cap W^{k,\infty})$ with $m \geq k+1 > \tfrac{n}{2}+2$ and  initial velocity
\begin{equation}\label{e:perturb}
	\bu_0 (x)= u_0(x) \bd +  \,v_0(x) \bd^{\star} \qquad \text{for some }  \bd,\bd^{\star}\in\S^{n-1},
\end{equation}
satisfying  
\begin{equation}\label{e:v0}
\inf_{x\in \T^n} e_0(x) \geq \sqrt{\e}, \quad \|u_0\|_{W^{1,\infty}} \approx 1, \quad  \|v_0\|_{W^{1,\infty}} \approx \e^2,
\end{equation}
for $\e$ small enough. Then the solution of \eqref{e:CSHydro} exists globally in time and is stable  around the underlying unidirectional motion:
\begin{equation}
	|\n v(t)|_\infty \lesssim \e, \quad \forall t>0.
\end{equation}
\end{theorem}

We note that perturbed solutions no longer fulfill conservation of entropy \eqref{e:elaw}, and as a result the technical challenge in proving \thm{t:stability} is to establish uniform control on the residual term that appears in the new entropy equation.  We will be able to perform such control for one dimensional perturbations of type \eqref{e:perturb}, and leave the general case for future research. 

The main application of our study lies in a construction of truly multi-dimensional solutions to the multi-scale alignment system introduced in \cite{ST-multi}.   The model is designed to describe behavior of cluster systems consisting of several interacting flocks. Alignment mechanisms within each flock are assumed to act on a faster time scale than between the flocks, and the agents of a given flock react to other flocks only though their group parameters. To write down the system we denote macroscopic flock variables by $(\rho_{\alpha},\bu_{\alpha})$ for $\alpha=1,\ldots, A$ and we define the following global flock parameters: center of masses
\[
 \mathcal{X}_{\alpha}(t):=\frac{1}{\mathcal{M}_{\alpha}}\int_{\R^n}x\,\rho_{\alpha}(x,t)\,dx \quad \text{where}  \quad \mathcal{M}_{\alpha}:=\int_{\R^n}\rho_{\alpha}(x,t)\dx,
\]
and momenta
\[
\mathcal{V}_{\alpha}(t):=\frac{1}{\mathcal{M}_{\alpha}}\int_{\R^n} \bu_{\alpha}(x,t)\rho_{\alpha}(x,t)\dx.
\]
The system represents as hybrid of hydrodynamic and discrete parts, where the hydrodynamic part corresponds to the Euler Alignment dynamics within each flock, while the discrete part governs inter-flock communication:
\begin{equation}\label{e:CSmultiflocks}
\left\{
\begin{split}
\partial_t \rho_{\alpha} +\nabla \cdot (\rho_{\alpha}  \bu _{\alpha})&= 0, \\
\partial_t  \bu _{\alpha}+ \bu _{\alpha}\cdot\nabla \bu _{\alpha}&=\l_\a [\phi_{\alpha}\ast (\rho_{\alpha}  \bu _{\alpha})- \bu _{\alpha}\left( \phi_\a \ast \rho_{\alpha}\right) ] +\varepsilon\sum_{\beta\neq\alpha}\cM_{\beta}\Psi(\mathcal{X}_{\alpha}-\mathcal{X}_{\beta})
\left(\mathcal{V}_{\beta}-\bu_{\alpha}\right),
\end{split}\right.
\end{equation}
where communication between flocks is assumed to be weaker than communication inside each of the flocks $\varepsilon\ll\min_{\alpha}\lambda_{\alpha}$.

\begin{figure}
	\centering
	\includegraphics[width=0.38\textwidth]{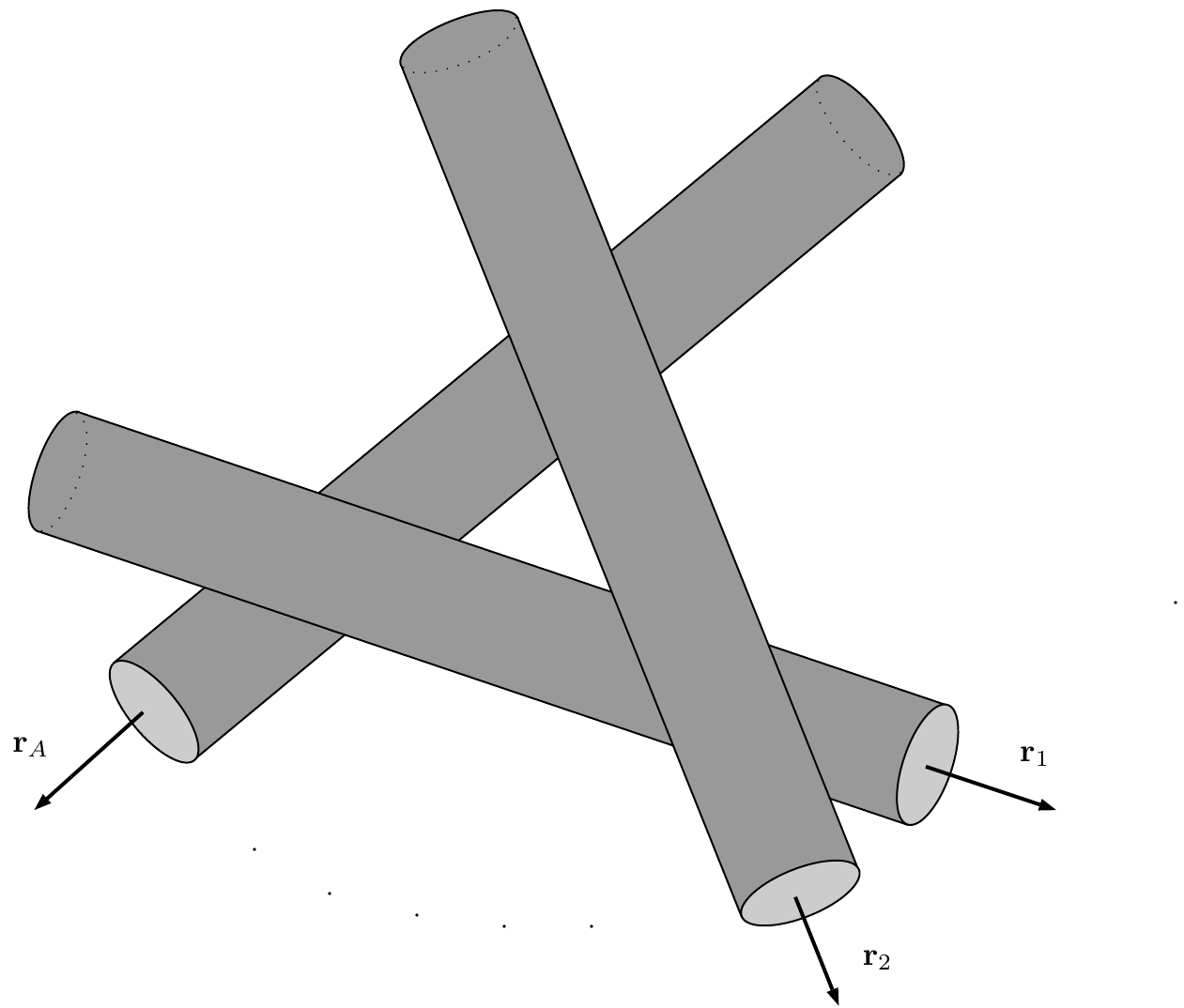}
	\caption{Mikado cluster in $v$-variables.}
	\label{f:mikado}
\end{figure}

Here $\phi_\a$'s are internal kernels and $\Psi$ is an inter-flock kernel, $\alpha=1,\ldots,A$. Integrating each momentum equation above one arrives at the upscaled system for global parameters $\cX_\a, \cV_\a$ given by the classical Cucker-Smale \eqref{e:CSagent} with $\phi = \Psi$.

As in the mono-flock case the system has a well-developed 1D theory, and exhibits flocking behavior on the prescribed time scales, see \cite{ST-multi}.  Using our  solutions as building blocks we can now construct new solutions to \eqref{e:CSmultiflocks} which form Mikado clusters -- by analogy with Mikado solutions to the 3D incompressible Euler equation which played crucial role in resolution of the celebrated Onsager conjecture, \cite{Daneri2017,Isett2018}. So, each block is oriented on average in the direction of its momentum with unidirectional variations:
\begin{equation}\label{e:mikadoansatz}
\bu_{\alpha}(x,t)=v_{\alpha}(x+\cX_\a(t),t)\,\br_{\alpha} + \cV_\a(t) \qquad \text{for} \quad v_{\alpha}:\mathbb{R}^n\times \mathbb{R}^{+}\rightarrow \mathbb{R}, \quad \br_\a \in \S^{n-1},
\end{equation}
 put together in a cluster formation.  Figure~\ref{f:mikado} illustrates the cluster configuration in new variables.

\begin{theorem}\label{t:fast&slow} Consider initial Mikado cluster \eqref{e:mikadoansatz} with $(v_\a(0),\rho_\a(0)) \in H^{m}\times (L_{+}^{1}\cap H^{k})$ with $m \geq k+1 > \tfrac{n}{2}+2$, satisfying the threshold condition $e_{\alpha}(0)\geq 0$ for all $\alpha=1,\ldots,A$. Then there exists a global in time unique solution to system \eqref{e:CSmultiflocks} which retains the same form \eqref{e:mikadoansatz} and satisfies $(v_\a,\rho_\a) \in C_w([0,\infty); H^m \times (L^1_+ \cap H^{k}))$. Moreover,
	\begin{itemize}
		\item{\em [Fast local flocking].} 	Assuming that for a given $\alpha\in\{1,\ldots,A\}$ the $\alpha$-flock has compact support and the internal  kernel $\phi_{\alpha}$ has a fat tail, then there exists $\delta_{\alpha}(\l_\a, \phi_{\alpha},\rho_{\alpha}(0),\bu_{\alpha}(0))$ such that
		\begin{align*}
		\underset{x\in \Supp{ \rho_{\alpha}(\cdot,t)}}{\sup}\left[ |\bu_{\alpha}(x,t)-\mathcal{V}_{\alpha}(t)|+|\nabla \bu_{\alpha}(x,t)|+|\nabla^2 \bu_{\alpha}(x,t)|\right]&\lesssim e^{-\delta_{\alpha}t},\\
		|\rho_{\alpha}(\cdot,t)-\bar{\rho}_{\alpha}(\cdot-\mathcal{X}_{\alpha}(t))|_{C^{\gamma}}&\lesssim e^{-\delta_{\alpha} t} \qquad (0<\gamma<1).
		\end{align*}
		\item{\em [Slow global flocking].}	Suppose the inter-flock kernel $\Psi$ has a fat tail and the internal kernels $\phi_{\alpha} \geq 0$ are arbitrary. If the multi-flock has a finite diameter initially, then global alignment occurs at a rate $\delta(\Psi,\varepsilon,\rho_{\alpha}(0),\bu_{\alpha}(0))$ such that
		\begin{align*}
		\underset{\alpha=1,\ldots,A}{\underset{x\in \Supp  \rho_{\alpha}(\cdot,t)}{\sup}}\left[ |\bu_{\alpha}(x,t)-\mathcal{V}|+|\nabla \bu_{\alpha}(x,t)|+|\nabla^2 \bu_{\alpha}(x,t)|\right]&\lesssim e^{-\delta t},\\[-0.3 cm]
		|\rho_{\alpha}(\cdot,t)-\bar{\rho}_{\alpha}(\cdot-\mathcal{V}\,t) |_{C^{\gamma}}&\lesssim e^{-\delta t} \qquad (0<\gamma<1),
		\end{align*}
		where $\cV = \frac{1}{\mathcal{M}} \sum_{\a=1}^A \mathcal{M}_\a \cV_\a$ is the global momentum.
	\end{itemize}
\end{theorem}


\noindent
{\em Organization.} In \sect{s:lwp} we prove a local existence  and a continuation criterium result in Sobolev spaces with minimal requirements needed for what follows. In \sect{s:multiD}, as a direct application of the continuation criterium, we obtain a global existence result for unidirectional parallel motion, and provide higher order control estimates on solutions to prove a strong flocking result.
In \sect{s:stability} we study a global existence result for almost unidirectional motion, and show its stability on $\mathbb{T}^{n}$.
Finally, in \sect{s:mikado} we discuss construction of Mikado clusters.

\section{Local existence and continuation criterion}\label{s:lwp}
In this section we build a local well-posedness theory and prove a new continuation criterion  for the Euler Alignment system \eqref{e:CSHydro} in classes suitable for subsequent flocking analysis.  We assume throughout that $\phi$ is sufficiently smooth to take as many derivatives as necessary in the course of our arguments below. We also assume that $\phi = \phi(x-y)$ is of convolution type and that the environment  is $\R^n$. The exact same results carry over to $\T^n$ ad verbatim.

\begin{theorem}[Local existence of classical solutions]\label{t:locsmooth}
	Assume that  $m \geq k+1> \frac{n}{2} + 2$ and $(u_0,\rho_0) \in H^m \times (H^k \cap L^1_+)$. Then there exists time $T_0 =T_0( |\n \bu_0|_\infty^{-1},\cM)$ and a unique solution to \eqref{e:CSHydro} on time interval $[0,T_0)$ in the class
	\begin{equation}\label{e:localclass}
	(\bu,\rho) \in C_w([0,T_0); H^m \times (H^k\cap L^1_+)) \cap \Lip ([0,T_0); H^{m-1} \times (H^{k-1} \cap L^1_+))
	\end{equation}
	satisfying the given initial condition. Moreover, any classical local solution on $[0,T_0)$ in class \eqref{e:localclass} and satisfying 
	\begin{equation}\label{e:BKMsmooth}
	\int_0^{T_0} \inf_{x\in \R^n} \n \cdot \bu(t,x) \dt > - \infty,
	\end{equation}
	can be extended beyond $T_0$.
\end{theorem}

One can obtain local existence rather easily for a viscous regularization:
\begin{equation}\label{e:CSHvisc}
\left\{
\begin{split}
\rho_t + \n \cdot (\rho \bu) & =  \e \D \rho, \\
\bu_t + \bu \cdot \n \bu &= \phi\ast( \bu \rho) - \bu\  (\phi \ast \rho) + \e \D \bu.
\end{split}\right.
\end{equation}
Indeed, denoting  $Z = (\bu, \rho)$ and we consider the equivalent mild formulation of \eqref{e:CSHvisc}
\[
Z(t)   = e^{\e t \D} Z_0 + \int_0^t  e^{\e (t-s) \D} \cN(Z(s)) \ds,
\]
where $\cN(Z)$ denotes all the nonlinear terms in \eqref{e:CSHvisc}.  The argument goes by the standard contractivity argument. Let us fix $Z_0 \in H^{m} \times (H^{k}\cap L_+^1): = X$ and consider the map
\[
\cT[Z](t) = e^{\e t \D} Z_0 + \int_0^t  e^{\e (t-s) \D} \cN(Z(s)) \ds.
\]
The argument to show that $T$ maps contractively $C([0,T); B_1(Z_0))$ into itself is elementary and relies on analyticity  of the heat semigroup:
\[
\| \n e^{\e t \D} f \|_{L^p} \lesssim \frac{1}{\sqrt{\e t}} \|f \|_{L^p}, \quad 1\leq p\leq \infty.
\]
So, by the fixed point argument we obtain a local solution on a time interval dependent on $\e$.  Denoting $T^*$ the maximal time of existence in $C([0,T); X)$ we show that $T^*$ depends only on the $X$-norm of the initial condition. We do it by establishing a priori estimates that are independent of $\e$ and  which  will allow us to pass to the limit of vanishing viscosity.  

So, the grand quantity we are trying to control is
\[
Y_{m,k} := \|\bu\|_{H^{m}}^2+  \|\rho\|_{H^k}^2 + \|\rho\|^2_1.
\]
To start with, we write the continuity equation as
\[
\rho_t + \bu \cdot \n \rho + (\n\cdot \bu) \rho =0.
\]
So, testing with $\p^{2k}\rho$ we obtain
\[ 
\ddt \|\rho\|_{\dH^k}^2 = \int (\n\cdot \bu) |\p^k \rho|^2 \dx - \int( \p^k(\bu \cdot \n \rho) - \bu \cdot \n \p^k \rho) \p^k \rho \dx - \int  \p^k ((\n\cdot \bu) \rho) \p^k \rho \dx - \e \|\rho\|^2_{H^{k+1}}.
\]
We dismiss the last term. Recalling the classical commutator estimate
\begin{equation}\label{e:classcomm}
\| \p^{k}(fg) - f \p^{k} g \|_2 \leq |\n f|_\infty \|g\|_{\dH^{k-1}} + \| f\|_{\dH^{k}} |g|_\infty,
\end{equation}
we obtain
\begin{multline*}
\ddt \|\rho\|_{\dH^k}^2 \leq  |\n \bu|_\infty  \|\rho\|_{\dH^k}^2 +  \|\bu\|_{\dH^k}  \|\rho\|_{\dH^k} |\n \rho|_\infty +  \|\bu\|_{\dH^{k+1}}  \|\rho\|_{\dH^k} |\rho|_\infty \\
\leq  C ( |\n \bu|_\infty  + |\n \rho|_\infty+ |\rho|_\infty) Y_{m,k},
\end{multline*}
provided $m \geq k+1$. The $L^2$-norm of $\rho$ obeys a similar estimate trivially, and the $L^1$-norm is conserved.

For the velocity equation we apply the same commutator estimate for the material derivative part:
\begin{multline*}
\int \p^{m} (\bu \cdot \n \bu) \p^{m} \bu \dx  = - \int \n \cdot \bu |\p^{m} \bu|^2 \dx + \int  [\p^{m} (\bu \cdot \n \bu) -(\bu \cdot \n \p^{m} \bu) ] \p^{m} \bu \dx\\ \lesssim |\n \bu|_\infty \|u\|_{\dH^{m}}^2.
\end{multline*}
For the alignment term, we can put all the derivatives onto the kernel whenever possible and the only term that is left out is $|\p^{m} \bu|^2 \, (\phi \ast \rho) $ with $\phi \ast \rho$ clearly bounded by $|\phi|_\infty \mathcal{M}$, a priori conserved quantity.   So, we obtain
\begin{equation*}\label{e:uapriori}
\ddt \|\bu\|_{\dH^{m}}^2 \leq (|\n \bu|_\infty + C(|\phi|_{C^m}, \mathcal{M}))\|\bu\|_{\dH^{m}}^2.
\end{equation*}
The similar bound for $\ddt \|\bu\|_{2}^2$ is derived trivially. So, we obtain
\begin{equation}\label{e:uapriori}
\ddt \|\bu\|_{H^{m}}^2 \leq (|\n \bu|_\infty +C(|\phi|_{C^m}, \mathcal{M})) \|\bu\|_{H^{m}}^2.
\end{equation}
It is important to note that this bound is independent of the higher norms of the density.
Combining the two equations we obtain
\begin{equation}\label{e:Yaux}
\ddt Y_{m,k} \leq C ( |\n \bu|_\infty  + |\n \rho|_\infty+ |\rho|_\infty + C(|\phi|_{C^m}, \mathcal{M})) Y_{m,k}.
\end{equation}
Of course $ |\n \bu|_\infty  + |\n \rho|_\infty+ |\rho|_\infty \leq Y_{m,k}$ provided $k> \frac{n}{2}+ 1$, which adds the last restriction on the exponents for the argument to work. So, if $m \geq k+1> \frac{n}{2} + 2$, then
\[
\ddt Y_{m,k} \leq C_1 Y_{m,k}+ C_2  Y^2_{m,k}.
\]
Solving the inequality gives a uniform bound  on a time interval inversely proportional to $\|Z_0\|_X$, but independent of $\e$. Thus, solutions to \eqref{e:CSHvisc} with the same initial data exist on a common time interval $[0,T_0]$ where they are uniformly bounded in $C([0,T_0];X)$. 

Let us also note that keeping the dissipative terms in the estimates above also shows that 
\[
\e \int_0^{T_0} ( \|\rho(s)\|^2_{H^{k+1}} +  \|\bu(s)\|^2_{H^{m+1}} ) \ds < C,
\]
where $C$ is independent of $\e$. Then 
\[
\left\| Z_t \right\|_{L^2} \leq \|Z\|_X^2 + \e \|Z\|_{H^2} \leq \|Z\|_X^2 + \e \|Z\|_{H^{k+1}\times H^{m+1}}.
\]
So, $ Z_t \in L^2([0,T_0];L^2)$. Passing to a subsequence we find a weak limit $Z_\e \to Z$ in $L^\infty([0,T_0];X)$ and  $ (Z_\e)_t \to Z_t$ in $L^2([0,T_0];L^2)$  (technically, a limit in $L^1$ may end up being a measure of bounded variation, however as a member of $H^k$ it is absolutely continuous, hence in $L^1$). Since we have that $Z_t \in L^2([0,T_0];L^2)$, $Z$ is weakly continuous with values in $L^2$. Since $L^2$ is dense in $H^{-m}$ and $H^{-k}$ this implies weak continuity $Z \in C_w([0,T_0]; H^m \times H^k)$. Strong continuity of the density follows from the equations itself:
\[
\| \rho_t \|_{L^1} \leq \|  \rho \n \bu\|_1 + \|\bu \n \rho\|_1 \leq \|Z\|_X^2 <C.
\]
Further regularity in time $Z_t$ follows from measuring smoothness of the system one level down and performing similar product estimates as above.

Having established local existence in $X$ let us come back to \eqref{e:Yaux} and notice that this solution can in fact be extended beyond $T_0$ if we knew 
\begin{equation}\label{e:BKM}
 \int_0^{T_0} |\n \bu(t)|_\infty \dt < \infty.
\end{equation}
Indeed,  \eqref{e:uapriori} implies that $\|\bu\|_{H^{m}}^2$ remains  bounded, and hence so is $|\n^2 \bu|_\infty$ since $m>\frac{n}{2}+ 2$. The norm $|\rho|_\infty$ can be bounded by solving the continuity equation along characteristics
\begin{equation}\label{e:rhochar}
\rho(X(t,\a),t) = \rho(\a,0) \exp\left\{- \int_0^t \n \cdot \bu(X(s,\a),s) \ds \right\}.
\end{equation}
Bootstrapping further by differentiating the continuity equation we bound $|\n \rho|_\infty$ in a similar fashion. This is sufficient to conclude that \eqref{e:Yaux} has a bounded factor on the right hand side.  

Having continuation criterion \eqref{e:BKM} we can determine the minimal time interval of existence by writing an equation for  $|\n \bu|_\infty$.  First, by the maximum principle, $|\bu(t) |_\infty \leq |\bu_0|_\infty$. Writing equation for one component $\p_i u_j$ we have
\[
\p_t \p_i u_j + \bu \cdot \n \p_i u_j + \p_i \bu \cdot \n u_j = \p_i \phi \ast (u_j \rho) - \p_i u_j (\phi \ast \rho) - u_j  (\p_i \phi \ast \rho ).
\]
Evaluating at the maximum and minimum and adding up over $i,j$ we obtain
\[
\ddt |\n \bu|_\infty \leq |\n \bu|_\infty^2 + C \mathcal{M} |\bu|_\infty + \mathcal{M} |\n \bu|_\infty.
\]
Hence, $|\n \bu|_\infty$ is uniformly bounded a priori on a time interval depending only on $|\n \bu_0|_\infty^{-1}$. So, the continuation criterion allows to extend our local solution up to the time $T_0 =T_0( |\n \bu_0|_\infty^{-1},\cM)$.

To relax \eqref{e:BKM} to \eqref{e:BKMsmooth} we first discuss Lagrangian formulation. Let us consider characteristics of the solution:
\[
\ddt X(t,\a) = \bu(X(t,\a),t), \quad X(0,\a) = \a.
\]
Denote $\bv(t,\a) = \bu(X(t,\a),t)$. Then \eqref{e:alignmentterm} can be written as
\begin{equation}
	\ddt \bv(t,\a) = \int_{\R^n} \phi(X(t,\a) - X(t,\b)) (\bv(t,\a) - \bv(t,\b)) \rho_0(\b) \db,
\end{equation}
here we used the transport property of the mass measure $\rho(t,\b) \db$. Let us now write the system for the deformation tensor of the flow map $(X,v)$. We have
\[
\ddt |\n X(t,\a)| \leq  |\n \bv(t,\a)|,
\]
and
\begin{multline*}
\ddt \n \bv(t,\a)  = \int_{\R^{n}} \n^\top X(t,\a) \n \phi(X(t,\a) - X(t,\b)) \otimes (\bv(t,\b) - \bv(t,\a) ) \rho_0(\b) \db  \\
- \n  \bv(t,\a)  \int_{\R^{n}}  \phi(X(t,\a) - X(t,\b)) \rho_0(\b)\db .
\end{multline*} 
By the maximum principle the amplitude $\sup_{\a,\b}|\bv(t,\b) - \bv(t,\a) |$ is decreasing. So,
\[
\ddt | \n \bv(t,\a) | \leq  C_1 |\n X(t,\a)| + C_2 | \n \bv(t,\a) | .
\]
This implies  exponential bound
\[
 | \n \bv(t) |_\infty + | \n X(t) |_\infty  \leq C_1 e^{C_2 t}.
\]
To obtain a bound in $|\n u|_\infty$ we invert back to labels 
\[
\n u(x,t) = \n^{-\top} X(X^{-1}(x,t),t) \n v(X^{-1}(x,t),t).
\]
Using that 
\[
| \n^{-\top} X(\a,t) | \leq \frac{C_1}{\inf_\a  |\det \n X(\a,t)|} e^{n C_2 t}
\]
we recall the Liouville formula
\[
\det \n X(\a,t) = \exp\left\{ \int_0^t \n \cdot \bu(X(\a,t),t) \dt \right\}.
\]
So, condition \eqref{e:BKMsmooth}  guarantees that the exponential will stay away from zero. This establishes uniform bound on $|\n \bu(t) |_\infty$.

\section{Existence and flocking of unidirectional solutions}\label{s:multiD}
Let us consider the following entropy for general multi-dimensional solutions of \eqref{e:CSHydro}:
\[
e = \n \cdot \bu + \phi \ast \rho.
\]
The equation for this quantity reads
\begin{equation}\label{e:eq}
\partial_t e+\n \cdot(\bu  e)=(\n \cdot \bu )^{2}-\Tr[(\nabla \bu )^{2}].
\end{equation}
It was derived in 2D in \cite{HeT2017}, and appeared briefly in the context of small initial data result in \cite{Shv2018}. Let us derive it in general for the sake of completeness. 

Since $\phi$ is a convolution kernel, we have that
\begin{equation}\label{e:eqaux1}
\partial_t (\phi \ast \rho)+\n \cdot \left[\phi \ast (\rho \, \bu)\right]=0.
\end{equation}
Taking the divergence of the velocity equation, we obtain
\begin{equation}\label{e:eqaux2}
\partial_t (\n \cdot \bu )+\n \cdot \left[ (\bu \cdot\n)\bu \right]=\n \cdot \left[\phi\ast (\rho  \bu ) \right]-\n \cdot\left[(\phi \ast \rho)\bu \right]
\end{equation}
with
$$\n \cdot[(\phi \ast\rho) \bu ]=(\phi \ast\rho) \n \cdot \bu + (\bu \cdot \n ) (\phi \ast\rho)$$
and
$$\n \cdot [( \bu \cdot\nabla) \bu ]=\Tr[(\nabla  \bu )^{2}]+ (\bu \cdot \n)(\n\cdot \bu)  .$$
On one hand, combining (\ref{e:eqaux1}) and (\ref{e:eqaux2}), we obtain that
\begin{equation}\label{e:eqaux3}
\partial_t e + (\phi \ast \rho) \n \cdot \bu + \bu \cdot\nabla e +\Tr[(\nabla  \bu )^{2}]=0.
\end{equation}
On the other hand, adding and subtracting now $(\n \cdot \bu)^2$ in \eqref{e:eqaux3} produces \eqref{e:eq}.
It is clear that in 1D the right hand side of \eqref{e:eq} vanishes, and one obtains a clean continuity law:
\begin{equation}\label{e:eclean}
	\partial_t e+\n \cdot(\bu  e) = 0.
\end{equation}
The main crucial observation we make here is that for the  unidirectional motion given by
\[
\bu(x,t)=u(x,t)\,\bd, \quad \bd \in \S^{n-1},
\]
the same cancelation occurs for the corresponding entropy
\[
e = \bd \cdot \n u + \phi \ast \rho,
\]
which satisfies the same equation \eqref{e:eclean}.

This opens  the possibility of addressing global existence in a way similar to the well-known one-dimensional case.

\begin{proof}[Proof of \thm{t:gwp}] Let us rewrite the $e$-equation as a non-autonomous logisitic ODE along characteristics:
\begin{equation}\label{e:log}
\ddt e =e (\phi \ast\rho-e ).
\end{equation}
It is clear that if $e_0(x_0) < 0$ for some $x_0$  then $e(t)<0$ along this particular characteristics, and hence we have 
\[
\ddt e =e (\phi \ast\rho-e ) < -e^2.
\]
This leads to blow up in finite time. 

On the other hand, if $e_0(x) \geq 0$ for all $x$, then $e(t,x) \geq 0$ remain so for all $t$ and $x$.  But then $\bd \cdot \n u  \geq - (\phi \ast \rho) \geq - |\phi|_\infty \cM$. This fulfills the continuation criterion \eqref{e:BKMsmooth} and the proof is complete.
\end{proof}

\begin{remark}
	To put our solutions in the context of the spectral dynamics approach of He and Tadmor \cite{HeT2017},
	we let $S_0:=\tfrac{1}{2}\left\lbrace \nabla \bu_0 +\nabla \bu_0^t \right\rbrace$ be the symmetric part of the velocity gradient of the initial field with eigenvalues $\mu_i=\mu_i(S_0)$. For our field it is given by $S_0 = \frac12( \bd \otimes \n u_0 + \n u_0 \otimes \bd)$. So, in dimension $2$ the spectral gap is given simply by
	\[
	\mu_2(S_0)-\mu_1(S_0) = | \n u_0|.
	\]
	This quantity, of course, can be arbitrarily large. Consequently, our solutions do not fulfill the threshold condition of \cite{HeT2017}.
\end{remark}

Next we address flocking  behavior. Due to rotational invariance of the system \eqref{e:CSHydro} one can assume without loss of generality that $\bd$ points in the direction of the $x_1$-axis:
\begin{equation}\label{e:ansatz(u)}
\bu(x ,t)=\langle u(x,t),0,\ldots,0\rangle \qquad \text{for} \quad u:\mathbb{R}^n\times \mathbb{R}^{+}\rightarrow \mathbb{R}.
\end{equation}
Let us now rewrite system \eqref{e:CSHydro} for the specific  ansatz \eqref{e:ansatz(u)} at hand:
\begin{equation}\label{e:uniCS}
(x,t)\in\mathbb{R}^n\times\mathbb{R}^{+}\qquad \left\{
    \begin{array}{@{} l l @{}}
      \partial_t \rho(x,t) +\partial_1(\rho\, u)(x,t)\hspace{-0.3 cm}&=0, \\
      \partial_t u(x,t)+ u(x,t)\partial_1 u(x,t) \hspace{-0.2 cm}&=\phi \ast(\rho u)(x,t)-u(x,t)(\phi \ast\rho)(x,t).
    \end{array}\right. 
\end{equation}
The entropy (\ref{entropy}) takes the form $e=\partial_1 u +\phi \ast \rho$ and the equation \eqref{e:eq} reduces to 
\[
\partial_t e+\partial_1(u e)=0.
\]
By the general result in multi-D proved in \cite{TT2014}, we have exponential alignment and flocking for any fat tail communication \eqref{e:ft}:
\begin{equation}\label{e:DA}
\begin{split}
\mathcal{D}(t)&\leq \bar{\cD}<\infty \qquad \text{where }\quad \mathcal{D}(t):=\underset{\left\lbrace x, y \right\rbrace\in\Supp \rho(\cdot,t)}{\text{max}}|x- y | \\
\cA(t) &\leq \cA_0 e^{-\d t}.
\end{split}
\end{equation}
Next, we complement this general result by a strong flocking statement of \thm{t:gwp}. This has so far been a 1D specific result, see \cite{ST2,ST3}, but we can extend it to the general multidimensional oriented flows and the use of the same entropy conservation. The technical issue in applying the 1D strategy is that, again, $e$ only controls $\p_1 u$. We will present a boostraping argument to extend such control to the full gradient and Hessian. 

\begin{proof}[Proof of strong flocking]

We know from \eqref{e:DA} that the diameter of the flock $\mathcal{D}(t)$ will remain finite. Then we can estimate the convolution from below:
\begin{equation}\label{e:actionlow}
(\phi\ast\rho)(\cdot,t)\geq \phi(\mathcal{D}(t))\cM\geq\phi(\bar{\cD})\cM,
\end{equation}
where we replace the kernel with its smallest value, which is $\phi(\bar{\cD})$. Using the logistic equation along characteristics obtained in \eqref{e:log} together with the previous bound \eqref{e:actionlow} we obtain the following ODI:
$$\ddt e=e(\phi \ast\rho-e)\geq e\left(\phi(\bar{\cD})\cM-e\right).$$
Since $e_0$ is uniformly bounded from above and away from zero, it follows that solving the logistic ODI there exists a
time $t^{\star} > 0$ such that   $e(x,t)\geq\frac{1}{2}\phi(\bar{\cD})\cM$ for all $x\in \Supp \rho(\cdot,t)$ and $t> t^{\star}$.

With this in mind let us write another equation for $\partial_i u$ in the following form
$$\partial_t \partial_i u+\partial_i u \partial_1 u +u \partial_i\partial_1 u=\partial_i\phi \ast(\rho u)-\partial_i u(\phi \ast\rho)-u(\partial_i\phi \ast\rho),$$
or along characteristics
\begin{align}\label{e:flockinggradu}
\ddt \partial_i u &=\partial_i\phi \ast(\rho u)-u(\partial_i\phi \ast\rho)-\left(\partial_1 u +\phi \ast\rho\right)\partial_i u \nonumber\\
&=\int_{\R^n} \partial_i \phi(|x- y |)\left(u( y )-u(x)\right)\rho( y )\,d y -e\,\partial_i u.
\end{align}
We already know from \eqref{e:DA} that the velocity fluctuations $\mathcal{A}(t)$ are exponentially decaying. Hence, the integral above will be bounded by $|\partial_i \phi|_{{\infty}}\cM E(t)$ where in what follows we denote by $E(t)$ a generic exponentially decaying quantity.

Recall that for large enough time, $t > t^{\star}$ we have the positive boundedness $e(x,t)\geq\frac{1}{2}\phi(\bar{\cD})\cM>0$ for all $x\in \text{Supp} \,\rho(\cdot,t)$. Then, evaluating (\ref{e:flockinggradu}) at
the maximum over $\Supp  \rho(\cdot,t)$ we obtain
$$\p_t |\p_i u|_{L^{\infty}(\Supp \rho(\cdot,t))}\leq E(t)-\tfrac{1}{2}\phi(\bar{\cD})\cM |\p_i u|_{L^{\infty}(\Supp \rho(\cdot,t))}.$$
This readily implies the desired result by integration, since at any time $t$ characteristics cover  $\Supp  \rho(\cdot,t)$ we arrive at the desired bound $|\partial_i u|_{L^{\infty}(\Supp \rho(\cdot,t))}\leq E(t)$ for $i=1,\ldots,n$. 

Moving on to the next level, let us note that $u \in W^{2,\infty}$ on $[0,T_0)$ by the Sobolev embedding. 
So, let us write the equation for the second order derivatives $|\partial_i\partial_j u|_{L^{\infty}}$ with ${i,j}=1,\ldots,n.$
$$\partial_t \partial_j\partial_i u+u\partial_j\partial_i\partial_1 u = \partial_j\partial_i\phi \ast(\rho u)-u(\partial_j\partial_i\phi \ast\rho) -(\partial_j\partial_i u+\partial_i u\partial_j+\partial_j u\partial_i) \left(\partial_1 u +\phi \ast \rho\right)$$
or along characteristics
$$
\ddt \partial_j\partial_i u =\int_{\R^n} \partial_j\partial_i \phi(|x- y |)\left(u( y )-u(x)\right)\rho( y )\,d y -e\,\partial_j\partial_i u-\partial_je\,\partial_iu-\partial_i e\,\partial_j u,$$
where 
$$\partial_j e=\partial_j\partial_1 u+\partial_j \phi \ast \rho \qquad \text{and} \qquad \partial_i e=\partial_i\partial_1 u+\partial_i \phi \ast \rho.$$
The procedure to prove exponential decay for second order derivatives consists in three steps.

\medskip
\noindent {\sc Step 1.}  First, we star considering the case $i=j=1.$ Note that for this particular case, we have
$$\ddt \partial_1^2 u =\int_{\Omega} \partial_1^2 \phi(|x- y |)\left(u( y )-u(x)\right)\rho( y )\,d y -e\,\partial_1^2 u-2\,\partial_1 e\,\partial_1 u.$$
Using that $\partial_1 e=\partial_1^{2} u+\partial_1 \phi \ast \rho,$ we arrive to the following inequality
$$\ddt  \partial_1^2 u  \leq E(t)-\partial_{1}^{2} u (e-E(t))$$
where $E(t)$ denotes a generic exponential decaying quantity. Now, as $e(x,t)\geq\frac{1}{2}\phi(\bar{D})\cM>0$ for all $x\in \text{Supp} \,\rho(\cdot,t)$ and $t>t^{\star}$ we have that
$$\ddt \partial_1^2 u \leq E(t)-\partial_{1}^{2} u \left(\tfrac{1}{2}\phi(\bar{\cD})\cM-E(t)\right) \qquad \text{for } t>t^{\star}.$$
As $E(t)$ decays exponentially fast, there must exists $t^{\star\star}>t^{\star}$ such that
$$
\ddt \partial_1^2 u  \leq E(t)-\tfrac{1}{4}\phi(\bar{\cD})\cM\partial_{1}^{2} u  \qquad \text{for } t>t^{\star\star}.$$
Then, evaluating the previous inequality at the maximum over $\Supp  \rho(\cdot,t)$ we have obtained the desired result by integration, since at any time $t$ characteristics cover all $\Supp  \rho(\cdot,t)$ we arrive at the desired bound $|\partial_1^{2} u|_{ L^{\infty}(\Supp \rho(\cdot,t)) }\leq E(t)$.

\medskip
\noindent {\sc Step 2.}  
Secondly, we consider the case $i=1$ and $i\neq j$. In this case, we have
$$\ddt \partial_{j}\partial_1 u  =\int_{\R^{n}} \partial_{j}\partial_1 \phi(|x- y |)\left(u( y )-u(x)\right)\rho( y )\,d y -e\,\partial_{j}\partial_1 u-\partial_j e\,\partial_1 u-\partial_{1}e\, \partial_{j} u,$$
	where 
$$\partial_j e=\partial_{j}\partial_1 u+\partial_j \phi \ast\rho \qquad \text{and} \qquad \partial_1 e=\partial_1^{2} u+\partial_1 \phi \ast\rho.$$
Using that $|\nabla u|_{L^{\infty}(\Supp \rho(\cdot,t))}\leq E(t)$ and the fact that $|\partial_{1}^{2} u|_{L^{\infty}(\Supp \rho(\cdot,t))}\leq E(t)$ we get 
$$\ddt  \partial_{j}\partial_1 u  \leq E(t)-\partial_{j}\partial_{1} u \left(e-E(t)\right)$$
and doing the same as before we obtain that $|\partial_{j}\partial_{1}u|_{L^{\infty}(\Supp \rho(\cdot,t))}\leq E(t)$ for $j\neq 1.$

\medskip
\noindent {\sc Step 3.}  
Finally, we consider the case $i\neq 1$. Using all the previous, we get 
$$\ddt \partial_j\partial_i u \leq E(t)-e\,\partial_j\partial_i u$$
and repeating the procedure we obtain that $|\partial_{j}\partial_{i}u|_{L^{\infty}}\leq E(t)$ for $i\neq 1$ and $j=1,\ldots,n$.

Consequently, we have proved that $|\nabla^{2} u|_{L^{\infty}(\Supp \rho(\cdot,t))}\leq E(t)$. Putting all together, there exists $C,\delta$  depending on $\phi$ and initial data $(\rho_0,u_0)$, such that the velocity satisfies
$$\underset{x\in \Supp \rho(\cdot,t)}{\sup}\left[|u(x,t)-\bar{u}|+|\nabla u(x,t)|+|\nabla^2 u(x,t)|\right]\leq C e^{-\delta t}.$$  
Moving to the density, we obtain solving the density equation along characteristics that
$$\rho(X(t),t)=\rho_{0}(X(0))\, \exp\left(-\int_{0}^{t}\partial_{1}u(X(s),s)\,ds\right).$$
So, in view of the established decay of $|\partial_{1} u|_{L^{\infty}}$, the density enjoys a pointwise global bound
$$\underset{t>0}{\text{sup}}\, |\rho(\cdot,t)|_{L^{\infty}}<\infty.$$
Next we establish a second round of estimates in higher order regularity in order to get a
control over $|\nabla \rho|_{L^{\infty}}$ and then prove flocking of the density. Let’s write the equation for $\partial_{i}\rho$ with $i=1,\ldots,n$.
$$\partial_{t}\partial_{i}\rho+\partial_{i}\partial_{1} u \,\rho+\partial_{1}u\,\partial_{i}\rho+\partial_{i}u\,\partial_{1}\rho+u\,\partial_{i}\partial_{1}\rho=0$$
or along characteristics
$$\ddt \partial_{i} \rho +\partial_{i}\partial_{1} u \,\rho+\partial_{1}u\,\partial_{i}\rho+\partial_{i}u\,\partial_{1}\rho=0.$$
The procedure to prove exponential decay consists in two steps:
\begin{enumerate}
	\item[\sc Step 1.] On one hand, we consider the case $i=1.$ Let’s write the equation for $\partial_{1}\rho$:
	$$\ddt  \partial_{1} \rho =-\partial_{1}^{2} u \,\rho-2\,\partial_{1}u\,\partial_{1}\rho=E(t)(1+\partial_{1}\rho).$$
	This shows that $|\partial_{1}\rho|_{L^{\infty}}$ is uniformly bounded. Afterthat, we are now ready to prove the same uniform bound for $\partial_{i}\rho$ with $i=2,\ldots,n.$
	
	\item[\sc Step 2.] On the other hand, we consider the case $i\neq 1.$  As $|\partial_{1}\rho|_{L^{\infty}}$ is uniformly bounded we have
	$$\ddt \partial_{i} \rho =-\partial_{i}\partial_{1} u \,\rho-\partial_{1}u\,\partial_{i}\rho-\partial_{i}u\,\partial_{1}\rho=E(t)(1+\partial_{i}\rho).$$
Therefore, this shows that $|\partial_{i}\rho|_{L^{\infty}}$ is uniformly bounded for $i=2,\ldots,n$.
\end{enumerate}
This shows that $|\nabla\rho|_{L^{\infty}}$ remains uniformly bounded. Now to establish strong 
flocking we have that the velocity alignment goes to its natural limit $\bar{u}=\mathcal{P}/M$.  Then, $\tilde{\rho}(x,t):=\rho(x_{1}+t\bar{u},x_{2},\ldots,x_n,t)$ satisfies
$$\partial_{t}\tilde{\rho}(x,t)+\tilde{\rho}(x,t)\,\partial_{1} u(x_{1}+t\bar{u},x_{2},\ldots,x_n,t) +\partial_{1}\tilde{\rho}(x,t)(u(x_{1}+t\bar{u},x_{2},\ldots,x_n,t)-\bar{u})=0.$$
According to the established bounds we have that $|\partial_{t}\tilde{\rho}|_{L^{\infty}}=E(t)$. This shows that $\tilde{\rho}(\cdot,t)$ is Cauchy in $t$ in the metric of $L^{\infty}$. Hence, there exists a unique limiting state $\rho_{\infty}(\cdot)$ such that $|\tilde{\rho}(\cdot,t)-\rho_{\infty}(\cdot)|_{L^{\infty}}=E(t).$
Shifting $x_{1}$ this can be expressed in terms of $\rho(x,t)$ and $\bar{\rho}(x,t):=\rho_{\infty}(x_{1}-t\bar{u},x_{2},\ldots,x_n,t)$ as 
$$|\rho(\cdot,t)-\bar{\rho}(\cdot)|_{L^{\infty}}= E(t).$$
Since $\nabla \rho$ is uniformly bounded, this also shows that $\bar{\rho}$ is Lipschitz. Convergence
in $C^{\gamma}$ with $0<\gamma<1$ follows by interpolation.
\end{proof}

\section{Stability and existence of flocks near a unidirectional one}\label{s:stability}
In this section we address the question of stability and prove \thm{t:stability}. We study oriented flocks in perturbative regime with initial condition
\[
\bu_0 (x)= u_0(x) \bd +  \,v_0(x) \bd^{\star} \qquad \text{for some }  \bd,\bd^{\star}\in\S^{n-1},
\]
with  $\| u_0\|_{W^{1,\infty}}\approx 1$ and $\|v_0\|_{W^{1,\infty}}\approx \e^2$ for some  $\e\ll 1$, which measures the size of the perturbation.  At the core of our analysis is again the entropy $e:=\n \cdot \bu+\phi\ast \rho$ which satisfies \eqref{e:eq} with non-zero right hand side.

The key element of our approach is to study an evolution equation for the expression on the right hand side and to establish control over its magnitude 
\[
\mathfrak{E}(t):=\|(\n \cdot \bu )^{2}-\Tr[(\nabla \bu )^{2}]\|_{L^{\infty}}.
\]
Note that initially $\mathfrak{E}(0)\approx \e^2$ and by continuity  $\mathfrak{E}(t)\lesssim \e^2$ at least for a short period of time. We therefore define a possible critical time $t^{\star}$ at which the solution hypothetically reaches size $\e$ for the first time:
\[
\mathfrak{E}(t^*)=\e, \qquad  \mathfrak{E}(t)<\e \quad \text{for } t<t^{\star}. 
\]
A contradiction will be achieved if we show that $ \mathfrak{E}'(t^{\star})<0$. This would establish the bound $\mathfrak{E}(t)<\e$ on the entire interval of existence, which in turn will imply a bound on $e$,  and hence extension to a global solution thanks to the continuation criterion \eqref{e:BKMsmooth}.

So, let derive the equation for $(\n \cdot \bu )^{2}-\Tr[(\nabla \bu )^{2}]$ to see what is needed to achieve a bound on it.  First, due to rotational invariance of the system \eqref{e:CSHydro} one can assume for simplicity that $\bd$ points in the direction of the $x_1$-axis and $\bd^{\star}$ points in the direction of the $x_2$-axis:
\begin{equation}\label{e:initialu}
	\bu_0(x) = \lan u_0(x), v_0(x),0,\ldots,0\ran.
\end{equation}
Then the solution at time $t$ takes form:
\begin{equation}\label{e:uv}
\bu(x,t) = \lan u(x,t), v(x,t),0,\ldots,0\ran.
\end{equation}
Let us observe the identity for general $\bw(x):=\langle w^1(x),\ldots,w^n(x)\rangle$
\[
(\n\cdot \bw)^{2}-\Tr[(\n \bw)^{2}]=\sum_{i\neq j}
\begin{vmatrix}
    \p_i w^i & \p_j w^i   \\
    \p_i w^j  & \p_j w^j  
\end{vmatrix}.
\]

\begin{definition}
Given three functions $f,g,h:\R^{n}\rightarrow \R$, we define the bracket $\{f,\cC_{\phi}(g,h)\}$ as another function  that takes the form
\[
\{f,\cC_{\phi}(g,h)\}:=\p_1f\, \cC_{\p_2 \phi}(g,h)-\p_2 f\, \cC_{\p_1 \phi}(g,h).
\]
\end{definition}
Using the structure of the Euler Alignment system \eqref{e:CSHydro} we derive the following equation along characteristics:
\begin{equation}\label{e:n=2}
\ddt \left[(\n\cdot\bu)^{2}-\Tr[(\n \bu)^{2}]\right]= -\left(e+\left(\phi\ast\rho\right)\right)\left[(\n\cdot\bu)^{2}-\Tr[(\n \bu)^{2}]\right]+\Upsilon_{\phi}(u,v)
\end{equation}
where 
$$\Upsilon_{\phi}(u,v):=2\left\lbrace u,\cC_{\phi}(\rho,v) \right\rbrace + 2\left\lbrace v,\cC_{\phi}(\rho,u) \right\rbrace.$$
Indeed, we have
\begin{align*}
\p_{t}(\p_{1}u)\,\p_{2}v+\left[(\partial_{1}\bu\cdot\n)u+(\bu\cdot\n)\p_{1}u\right]\p_{2}v&=\p_{2}v\,\cC_{\p_{1}\phi}(\rho,u)-\p_{2}v\,\p_{1}u \left(\phi\ast\rho\right),\\
\p_{1}v\,\p_{t}(\p_{2}u)+\left[(\p_{2}\bu\cdot\n)u+(\bu\cdot\n)\p_{2}u\right]\p_{1}v&=\p_{1}v\,\mathcal{C}_{\p_{2}\phi}(\rho,u)-\p_{1}v\,\p_{2}u \left(\phi\ast\rho\right),\\
\p_{1}u\,\p_{t}(\p_{2}v)+\left[(\p_{2}\bu\cdot\n)v+(\bu\cdot\n)\p_{2}v\right]\p_{1}u&=\p_{1}u\,\cC_{\p_{2}\phi}(\rho,v)-\p_{1}u\,\p_{2}v \left(\phi\ast\rho\right),\\
\p_{t}(\p_{1}v)\,\p_{2}u+\left[(\p_{1}\bu\cdot\n)v+(\bu\cdot\n)\p_{1}v\right]\p_{2}u&=\p_{2}u\,\mathcal{C}_{\p_{1}\phi}(\rho,v)-\p_{2}u\,\p_{1}v \left(\phi\ast\rho\right),
\end{align*}
and consequently, 
\begin{align}\label{e:cancelationPSI}
\ddt \begin{vmatrix}
    \p_1 u & \p_2 u   \\
    \p_1 v  & \p_2 v  
\end{vmatrix}+ \Psi(\bu,\n\bu)= -2\,\begin{vmatrix}
    \p_1 u & \p_2 u   \\
    \p_1 v  & \p_2 v  
\end{vmatrix}\left(\phi\ast\rho\right)+\Upsilon_{\phi}(u,v)
\end{align}
where
$$\Psi(\bu,\n\bu):=\p_{2}v\left[(\p_{1}\bu\cdot\n)u\right]+\p_{1}u\left[(\p_{2}\bu\cdot\n)v\right]-\p_{2}u\left[(\p_{1}\bu\cdot\n)v\right]-\p_{1}v\left[(\p_{2}\bu\cdot\n)u\right].$$
Note that, combining the terms in $\Psi(\bu,\n\bu)$ appropriately give us a new cancellation:
\begin{align*}
\p_{2}v\left[(\p_{1}\bu\cdot\n)u\right]-\p_{2}u\left[(\p_{1}\bu\cdot\n)v\right]&=\p_{2}v  \left(\p_{1}u\right)^{2}-\p_{2}u \, \p_{1}u\, \p_{1}v,\\
\p_{1}u\left[(\p_{2}\bu\cdot\n)v\right]-\p_{1}v\left[(\p_{2}\bu\cdot\n)u\right]&=\p_{1}u\left(\p_{2}v\right)^{2}-\p_{1}v\, \p_{2}v\, \p_{2}u,
\end{align*}
which implies
\begin{equation}\label{e:PSI=det}
\Psi(\bu,\n\bu)=\begin{vmatrix}
    \p_1 u & \p_2 u   \\
    \p_1 v  & \p_2 v  
\end{vmatrix}(\n\cdot \bu).
\end{equation}
Putting together \eqref{e:cancelationPSI} and \eqref{e:PSI=det} with the fact 
that $e=\n\cdot\bu +(\phi\ast\rho)$ we have proved \eqref{e:n=2}.

\begin{lemma}\label{l:e} On the interval $[0,t^*]$ we have
	\begin{equation}\label{e:eest}
	\tfrac{1}{2}\sqrt{\e} \leq e(x,t)\leq C_0: =2  \max\{ |e_0|_\infty , \cM |\phi|_\infty\}.
	\end{equation}
\end{lemma}

\begin{proof}Recall the equation:
	\begin{equation}\label{e:lowerentropy}
	\ddt e=\left[(\n\cdot\bu)^{2}-\text{Tr}[(\n \bu)^{2}]\right]+e\left((\phi\ast\rho)-e\right).
	\end{equation} 
If $e(x,t) = \tfrac{1}{2}\sqrt{\e}$ for the first time $t<t^*$, then 
	\[
	\ddt e \geq -\e + \tfrac{1}{2}\sqrt{\e}( \phi \ast \rho(x,t) - \tfrac{1}{2}\sqrt{\e}) \geq -\e+ \tfrac{1}{2}\sqrt{\e}( \cM \inf_{x} \phi - \tfrac{1}{2}\sqrt{\e}) >0,
	\]
provided $\e \lesssim \cM^2$, a contradiction.  Similarly, if $e(x,t) = C_0 > |e_0|_\infty$, then 
\[
\ddt e \leq \e + C_0(|\phi|_\infty \cM - C_0) < 0,
\]
provided $C_0 > 2 \cM |\phi|_\infty$. So, the constant can be chosen $C_0 =2  \max\{ |e_0|_\infty , \cM |\phi|_\infty\}$. 
\end{proof}

We now proceed to establishing control over the partial gradient $|\n_{1,2} \bu(t) |_{{\infty}}$ which is what is needed to bound the residual term $\Upsilon_{\phi}(u,v)$ in \eqref{e:n=2}.  First, we derive an equation for the full gradient in the form suitable for our application, see also \cite{HeT2017}. 

\begin{lemma}
Assume that $\bw$ solves the multi-dimensional Euler Alignment system 
\[
\p_t \bw +(\bw\cdot\n)\bw=\mathcal{C}_{\phi}(\rho,\bw).
\]
 Then, the following equality holds:
\begin{equation}\label{e:genralD}
\ddt \p_l w^{k}-\sum_{i=1}^{n}\begin{vmatrix}
    \p_i w^{i} & \p_l w^{i}   \\
    \p_i w^{k}  & \p_l w^{k}  
\end{vmatrix}=\cC_{\p_l \phi}(\rho, w^{k})-e\,\p_l w^{k}.
\end{equation}
\end{lemma}
\begin{proof}
Computing the $\p_l$ derivative, we get componentwise
\[
\p_t \p_l w^{k}+(\p_l\bw\cdot\n) w^{k}+(\bw\cdot\n) \p_l w^{k}=\cC_{\p_l \phi}(\rho, w^{k})-\p_l w^{k} (\phi\ast \rho).
\]
Therefore, we obtain that
\begin{align*}
\ddt  \p_l w^{k} &=\cC_{\p_l \phi}(\rho, w^{k})-\p_l w^{k} (\phi\ast \rho)-(\p_l\bw\cdot\n) w^{k}\\
&=\cC_{\p_l \phi}(\rho, w^{k})-e\,\p_l w^{k}+\p_l w^{k}(\n\cdot \bw)-(\p_l\bw\cdot\n) w^{k},
\end{align*}
where
\[
\p_l w^{k}(\n\cdot \bw)-(\p_l\bw\cdot\n) w^{k}=\sum_{i=1}^{n}\begin{vmatrix}
    \p_i w^{i} & \p_l w^{i}   \\
    \p_i w^{k}  & \p_l w^{k}  
\end{vmatrix}.
\]
\end{proof}

For our solutions the gradient of the velocity field takes form 
\[
\n\bu=\left(\begin{array}{@{}c|c@{}}
  \begin{matrix}
  \p_1 u & \p_2 u \\
  \p_1 v & \p_2 v
  \end{matrix}
  & \begin{matrix}
  \p_3 u & \ldots & \p_n u \\
  \p_3 v & \ldots & \p_n v
  \end{matrix} \\
\hline
  \begin{matrix}
  0 & 0 \\
  \vdots & \vdots\\
  0 & 0
  \end{matrix} &
  \begin{matrix}
  0 & \ldots &0 \\
  \vdots & \ddots &\vdots\\
  0 & \ldots &0
  \end{matrix}
\end{array}\right)\equiv \left(\begin{array}{@{}c|c@{}}
  P & Q \\
\hline
  0 & 0
\end{array}\right)
\]
So, we first address the upper corner part $P$.

\begin{lemma}\label{e:off-diagonal}
On time interval $[0,t^*]$ we have the following bounds
\begin{equation}\label{e:W1}
	|\p_1 v(t)|_{{\infty}}+|\p_2 v(t)|_{{\infty}}\leq C \e^{3/2}, \qquad  |\partial_{1}u(t)|_{{\infty}}+ |\partial_{2}u(t)|_{{\infty}} \leq C \e^{-1/2}.
\end{equation}
\end{lemma}
\begin{proof} 
	First, we consider the off-diagonal elements of submatrix $P$:
	\begin{equation}
	\left\{
	\begin{split}
	\ddt \p_2 u&=\cC_{\p_2 \phi}(\rho, u)-e\,\p_2 u, \\
	\ddt \p_1 v &=\cC_{\p_1 \phi}(\rho, v)-e\,\p_1 v.
	\end{split}\right.
	\end{equation}
Recalling \eqref{e:eest} we have
\begin{align*}
\p_{t}|\p_{1}v|_\infty &\leq 2\,|\p_{1}\phi|_\infty\cM |v_0|_\infty- \frac12 \sqrt{\e}\, |\p_{1}v|_\infty,\\
\p_{t}|\p_{2}u|_\infty&\leq 2\,|\p_{2}\phi|_\infty \cM |u_0|_\infty - \frac12 \sqrt{\e}\,|\p_{2}u|_\infty.
\end{align*}
Gronwall's inequality tell us that
\begin{align*}
|\p_{1}v|_\infty &\leq |\p_{1}v_0|_\infty \,e^{-\frac12 \sqrt{\e} t}+\frac{2}{\sqrt{\e}}\,|\p_{1}\phi|_\infty\cM |v_0|_\infty (1-e^{-\frac12 \sqrt{\e} t}) \leq C \e^{3/2}\\
|\p_{2}u|_\infty&\leq |\p_{2}u_0|_\infty\,e^{-\frac12 \sqrt{\e} t}+\frac{2}{\sqrt{\e}}\,|\p_{2}\phi|_\infty\cM |u_0|_\infty(1-e^{-\frac12 \sqrt{\e} t})\leq C \e^{-1/2}.
\end{align*}
For the diagonal entries, we have
\begin{equation}\label{e:systemP}
\left\{
\begin{split}
\ddt \p_1 u -\tfrac{1}{2}\left[(\n\cdot\bu)^{2}-\text{Tr}[(\n \bu)^{2}]\right]&=\cC_{\p_1 \phi}(\rho, u)-e\,\p_1 u, \\
\ddt  \p_2 v -\tfrac{1}{2}\left[(\n\cdot\bu)^{2}-\text{Tr}[(\n \bu)^{2}]\right]&=\cC_{\p_2 \phi}(\rho, v)-e\,\p_2 v.
\end{split}\right.
\end{equation}
Thus, on the interval $[0,t^*]$:
\[
\p_{t}|\p_{1}u|_\infty \leq \e + 2\,|\p_{1}\phi|_\infty\cM |u_0|_\infty- \frac12 \sqrt{\e}\, |\p_{1}u|_\infty,
\]
Hence,
\[
|\p_{1}u|_\infty \leq  |\p_{1}u_0|_\infty \,e^{-\frac12 \sqrt{\e} t}+\frac{2}{\sqrt{\e}}\,(|\p_{1}\phi|_\infty\cM |u_0|_\infty +\e) (1-e^{-\frac12 \sqrt{\e} t}) \leq C\e^{-1/2}.
\]
A similar estimate for $\p_2 v $ would have given us $\lesssim \e^{1/2}$, which is not sufficient. 
To improve this bound, we need to rewrite the equation for $\p_2 v$ in a different way:
\begin{equation}\label{e:d2v}
\ddt  \p_2 v +\p_2 u \, \p_1 v+ \p_2 v\p_2 v=\cC_{\p_2 \phi}(\rho, v)-(\phi\ast\rho)\,\p_2 v.
\end{equation}
To analyze the evolution in time of $\p_2 v$ through \eqref{e:d2v}, we define the points $x^{\pm}(t)$ as follows
$$\p_2 v(x^{+}(t),t):=\max \p_2 v(x,t) \qquad \text{and} \qquad \p_2 v(x^{-}(t),t):=\min \p_2 v(x,t).$$
Then,
\begin{align*}
\ddt  \p_2 v(x^{+}(t),t)&\leq -\left[(\phi\ast\rho)+\p_2 v (x^{+}(t),t)\right]\,\p_2 v(x^{+}(t),t)+ C \e^{3/2},\\
- \ddt \p_2 v(x^{-}(t),t)&\leq \left[(\phi\ast\rho)+\p_2 v(x^{-}(t),t)\right]\,\p_2 v(x^{-}(t),t)+C \e^{3/2},
\end{align*}
and their difference $\mathrm{d}(t):=\p_2 v^{\e}(x^{+}(t),t)-\p_2 v^{\e}(x^{-}(t),t)$ 
satisfies
\begin{equation}\label{e:logisticD}
\mathrm{d}'(t)\leq -\left[(\phi\ast\rho)+\p_2 v^{\e}(x^{+}(t),t)-\p_2 v^{\e}(x^{-}(t),t)\right] \mathrm{d}(t)+C \e^{3/2}.
\end{equation}
We already know that $|\p_2 v |_\infty \leq \e^{1/2}$ we have 
\[
(\phi\ast\rho)+\p_2 v^{\e}(x^{+}(t),t)-\p_2 v^{\e}(x^{-}(t),t) \geq c_0,
\]
and hence
\[
\mathrm{d}'(t)\leq  - c_0\mathrm{d}(t) + C \e^{3/2}, \quad \mathrm{d}(0) \leq \e^2.
\]
Application of Gr\"onwall's lemma gives
$\mathrm{d}(t) \lesssim \e^{3/2}$ and the proof is complete.
\end{proof}

With all ingredients at hand we are now ready to use the equation \eqref{e:n=2} to close the argument for global existence. At the critical time $t^{\star}$, we have 
\[
\ddt\mathfrak{E}(t^{\star})\leq -c \cM \mathfrak{E}(t^{\star})+\Upsilon_{\phi}(u,v)
\]
where 
\[
\Upsilon_{\phi}(u,v)=2\left\lbrace u,\cC_{\phi}(\rho,v) \right\rbrace + 2\left\lbrace v,\cC_{\phi}(\rho,u) \right\rbrace.
\]
We have, using \eqref{e:W1} and $|v|_\infty \leq |v_0|_\infty \approx \e^2$,
\begin{align*}
\left\lbrace u,\cC_{\phi}(\rho,v) \right\rbrace & \leq  C \e^{-1/2} \e^{2} = C \e^{3/2} \\
\left\lbrace v,\cC_{\phi}(\rho,u) \right\rbrace & \leq  C \e^{3/2}.
\end{align*}
Then 
\[
\ddt \mathfrak{E}(t^{\star})\leq -c \e +C\e^{3/2} <0
\]
for $\e>0$ small enough.  This shows that $t^* = \infty$, and hence by \lem{l:e} we have a uniform bound on $e$, which fulfills the continuation criterion of \thm{t:locsmooth}.

Finally, we establish control over the remaining part of the gradient matrix $Q$.  
\begin{lemma}
We have  for all $k = 3,\ldots,n$ and all time $t>0$,
\[
|\p_k u(t)|_\infty \leq c \e^{-1/2}, \qquad |\p_k v(t)|_\infty  \leq c \e.
\]
\end{lemma}
\begin{proof}
We write the system for $\p_k u $ and $\p_k v$ as follows:
\begin{equation}\label{e:goodsystem}
\left\{
\begin{split}
\ddt \p_k u - \p_2 v\, \p_k u + \p_2 u\, \p_k v&=\cC_{\p_k \phi}(\rho, u)-e\,\p_k u, \\
\ddt \p_k v +\p_1 v\, \p_k u +\p_2 v\, \p_k v&=\cC_{\p_k \phi}(\rho, v)- (\phi\ast \rho)\,\p_k v. \\
\end{split}\right.
\end{equation}
Denote $X_k(t):=|\p_k u(t)|_\infty$ and $Y_k(t):=|\p_k v(t)|_\infty$ for $k=3,\ldots,n$.
Combining \eqref{e:W1} together with the lower bound $(\phi\ast\rho)\geq c_0$ and the fact that $e(t)\geq \frac12 \sqrt{\e}$, the system \eqref{e:goodsystem} can be put as
\begin{equation}\label{e:matrixsystem}
\left\{
\begin{split}
\dot{X}_k(t)&\leq -c_1\sqrt{\e} X_k(t) + \frac{c_2}{\sqrt{\e}} Y_k(t)  + c_3, \\
\dot{Y}_k(t)&\leq c_4 \e^{3/2} X_k(t) - c_5 Y_k(t)+ c_6 \e^2
\end{split}\right.
\end{equation}
where all the constants are independent of $\e$. Obviously, defining the vector $Z_k(t):=(X_k(t),Y_k(t))$ the system \eqref{e:matrixsystem} of ODI can be re-write in matrix form as $\dot{Z}_k(t)\leq \mathrm{A} Z_k(t)+\mathrm{b}$ with diagonalization $\mathrm{A}=\mathrm{P D P}^{-1}$, where
\[
\mathrm{A}:=\begin{pmatrix}
-c_1 \sqrt{\e} & \frac{c_2}{\sqrt{\e}}\, \\
c_4 \e^{3/2} & -c_5
\end{pmatrix}, \quad \mathrm{b}:=\begin{pmatrix}
c_3 \, \\
c_6 \e^{2} 
\end{pmatrix}
\qquad \text{and} \qquad \mathrm{D}:=\begin{pmatrix}
\lambda_{+}(\e) & 0\, \\
0 & \lambda_{-}(\e)
\end{pmatrix}.
\]
Noting that $\text{Tr}(\mathrm{A})<0$ and $\text{Det}(\mathrm{A})>0$ for $\e>0$ small enough, both eigenvalues $\lambda_{\pm}(\e)$ are negative.
More specifically, we have that
\[
\lambda_{\pm}(\e):=\tfrac{\text{Tr}(\mathrm{A})}{2}\pm \sqrt{\left(\tfrac{\text{Tr}(\mathrm{A})}{2}\right)^2-\text{Det}(\mathrm{A})}
\]
with
\[
\lambda_{+}(\e)\to -c_5+\cO(\e) \qquad \text{and} \qquad \lambda_{+}(\e)\to -\sqrt{\e}+\cO(\e).
\]
Making an ansatz to use an integrating factor of $e^{-\mathrm{A} t}$ and multiplying throughout, yields
\[
\mathrm{Z}_k(t)\leq e^{\mathrm{A}t}Z_k(0)+\int_{0}^{t}e^{\mathrm{A}(t-s)}\mathrm{b}\,ds.
\]
So, calculating $e^{\mathrm{A}t}=\mathrm{P}e^{\mathrm{D}t}\mathrm{P}^{-1}$ leads to the solution to the system, by simply integrating  with respect to t. Using that $Z_k(0)\approx (1,\e^2)$ and some elementary linear algebra, we obtain that
\begin{align*}
X_k(t)&\lesssim \left( e^{\lambda_{-}t}+e^{\lambda_{+}t} \right)+\left(\frac{e^{\lambda_{-}t}-1}{\lambda_{-}}+\frac{e^{\lambda_{+}t}-1}{\lambda_{+}}\right),\\
Y_k(t)&\lesssim \e^{3/2} \left( e^{\lambda_{-}t}-e^{\lambda_{+}t}\right) +\e^{3/2}\left(\frac{e^{\lambda_{-}t}-1}{\lambda_{-}}-\frac{e^{\lambda_{+}t}-1}{\lambda_{+}}\right).
\end{align*}
Combining everything we conclude that  
\[
X_k(t) \leq  c_7  \e^{-1/2}, \quad Y_k(t) \leq c_8 \e, \quad \forall t\geq 0,
\]
as desired.
\end{proof}
Combined with previously established bound \eqref{e:W1} we conclude the stability bound for the full gradient of the perturbation $|\n v(t)|_\infty \lesssim \e$.

\section{Mikado clusters}\label{s:mikado}

We now turn our attention to the cluster system \eqref{e:CSmultiflocks}.  First issue we encounter is the maximum principle. Although the system satisfies the global maximum principle -- each velocity component variation is non-increasing, this may not be the case within each individual flock. One can derive an ``internal maximum principle'', meaning that the velocities relative to the momentum of $\alpha$-flock are in fact decaying. To obtain this let us pass to the reference frame moving with the average momentum in each flock:
$$\mathbf{v}_{\alpha}(x,t):=\bu_{\alpha}(x-\mathcal{X}_{\alpha}(t),t)-\mathcal{V}_{\alpha}(t)\qquad \text{and} \qquad \varrho_{\alpha}(x,t):=\rho_{\alpha}(x-\mathcal{X}_{\alpha}(t),t).$$
and write the system in the new variables
\begin{equation}\label{e:CShydroW}
 \left\{
    \begin{split}
      \partial_t \varrho_{\alpha} +\n\cdot(\varrho_{\alpha} \textbf{v}_{\alpha})&= 0, \\
      \partial_t \textbf{v}_{\alpha}+\textbf{v}_{\alpha}\cdot\nabla\textbf{v}_{\alpha}&=\lambda_{\alpha} \left[\phi_{\alpha} \ast(\varrho_{\alpha} \textbf{v}_{\alpha})-\textbf{v}_{\alpha}(\phi_{\alpha} \ast\varrho_{\alpha})\right]+\varepsilon R_{\alpha}(t)\mathbf{v}_{\alpha},
    \end{split}\right.
\end{equation}
where
$$R_{\alpha}(t):=\sum_{\beta\neq\alpha}\cM_{\beta}\Psi(\mathcal{X}_{\alpha}(t)-\mathcal{X}_{\beta}(t)).$$
The following statement is easy to verify.
\begin{lemma}
	The set of variables $(\bu_\a,\rho_\a)_\a$ satisfy \eqref{e:CSmultiflocks} if and only if $(\cX_\a,\cV_\a)_\a$ satisfy the discrete Cucker-Smale system \eqref{e:CSagent}, and $(\bv_\a,\varrho_\a)_\a$ satisfy \eqref{e:CShydroW}.
\end{lemma}
Moreover, one can prove a similar local well-posedness result as in \thm{t:locsmooth} with similar continuation criterion satisfied by each $\a$-flock.

So, we construct Mikado solutions to \eqref{e:CSmultiflocks} by specifying an arbitrary set of macroscopic parameters $(\cX_\a,\cV_\a)_\a$ satisfying the discrete Cucker-Smale system \eqref{e:CSagent} and setting 
\begin{equation}\label{e:mikadoansatz2}
\bv_{\alpha}(x,t)=v_{\alpha}(x,t)\,\br_\a \qquad \text{for} \quad v_{\alpha}:\mathbb{R}^n\times \mathbb{R}^{+}\rightarrow \mathbb{R}, \quad \br_\a\in \S^{n-1}.
\end{equation}
As in the monoflock case the entropy plays a crucial role,
\begin{equation}\label{e:ealpha}
e_\a := \br_\a \cdot \n v_{\a}+\lambda_{\a}(\phi_{\a} \ast\varrho_{\a}),
\end{equation}
which satisfies
\begin{equation*}
\partial_t e_{\alpha}+  \n \cdot ( \bv_{\a} \varrho_\a) = -\varepsilon R_{\alpha}(t)\,(\n\cdot\mathbf{v}_{\alpha}),
\end{equation*}
or equivalently along characteristics
\[
\ddt  e_{\alpha}=(\varepsilon R_{\alpha}(t)+e_{\alpha})(\lambda_{\alpha}(\phi_{\a} \ast\varrho_{\alpha})-e_{\alpha}).
\]
Since $R_\a \geq 0$, the initial positive entropy $e_\a \geq 0$ will preserve its sign, and also be globally bounded. Thus, $\n \cdot \bv_\a$ is bounded, and hence we obtain global existence by \thm{t:locsmooth}.

It was already shown in \cite{ST-multi} that any classical solution to a multi-flock aligns exponentially fast.   To prove strong flocking we simply observe that the scalar pair $(v_\a,\varrho_\a)$  satisfies 
\begin{equation}\label{e:CSmultinew}
\left\{
\begin{split}
\partial_t \varrho_{\alpha} +\n\cdot(\varrho_{\alpha} v_{\alpha}\textbf{r}_{\alpha})&= 0, \\
\partial_t v_{\alpha}+\left(\textbf{r}_{\alpha}\cdot\nabla v_{\alpha}\right)v_{\alpha}&=\lambda_{\alpha} \left[\phi_{\alpha} \ast(\varrho_{\alpha} v_{\alpha})-v_{\alpha}(\phi_{\alpha} \ast\varrho_{\alpha})\right]+\varepsilon R_{\alpha}(t)v_{\alpha},
\end{split}\right.
\end{equation}
which is similar to \eqref{e:uniCS} with the exception of the extra term $\varepsilon R_{\alpha}(t)v_{\alpha}$ which plays the role of extra damping since $R_\a \geq 0$.  So, the same analysis as in mono-flock case applies.

It is interesting to observe that each $\alpha$-flock aligns to its momentum $\mathcal{V}_{\alpha}$ regardless of whether momenta themselves align or not. So, strong internal communication $\phi_{\alpha}$ leads to local emergent behavior within
the  $\alpha$-flock despite potentially destabilizing influence of the others flocks. On the other hand, if the inter-communication kernel $\Psi$ is global, then the global emergence occurs even if internal communications are weak or completely absent. In this case, every agent aligns to the total momentum of the system $\mathcal{V}$.\\


\begin{thebibliography}{10}
	
	\bibitem{CCP2017}
	Jos\'e~A. Carrillo, Young-Pil Choi and Sergio P.~Perez.
	\newblock A review on attractive-repulsive hydrodynamics for consensus in
	collective behavior.
	\newblock In {\em Active Particles}, Volume~1. Birkhäuser, 2017.
	
	\bibitem{CCTT2016}
	Jos\'e~A. Carrillo, Young-Pil Choi, Eitan Tadmor and Changhui Tan.
	\newblock Critical thresholds in 1{D} {E}uler equations with non-local forces.
	\newblock {\em Math. Models Methods Appl. Sci.}, 26(1):185--206, 2016.
	
	\bibitem{CCMP2017}
	Jos\'e~A. Carrillo, Young-Pil Choi, Piotr~B. Mucha and Jan Peszek.
	\newblock Sharp conditions to avoid collisions in singular {C}ucker-{S}male
	interactions.
	\newblock {\em Nonlinear Anal. Real World Appl.}, 37:317--328, 2017.
	
	\bibitem{CS2007a}
	Felipe Cucker and Steve Smale.
	\newblock Emergent behavior in flocks.
	\newblock {\em IEEE Trans. Automat. Control}, 52(5):852--862, \nolinebreak 2007.
	
	\bibitem{CS2007b}
	Felipe Cucker and Steve Smale.
	\newblock On the mathematics of emergence.
	\newblock {\em Jpn. J. Math.}, 2(1):197--227, 2007.
	
	\bibitem{Daneri2017}
	Sara Daneri and L\'{a}szl\'{o} Sz\'{e}kelyhidi, Jr.
	\newblock Non-uniqueness and h-principle for {H}\"{o}lder-continuous weak
	solutions of the {E}uler equations.
	\newblock {\em Arch. Ration. Mech. Anal.}, 224(2):471--514, 2017.
	
	\bibitem{FK2017}
	Alessio Figalli and Moon-Jin Kang.
	\newblock A rigorous derivation from the kinetic Cucker-Smale model to the
	pressureless Euler system with nonlocal alignment.
	\newblock {\em  Anal. PDE}, 12(3): 843--866, 2019.
	
	\bibitem{HL2009}
	Seung-Yeal Ha and Jian-Guo Liu.
	\newblock A simple proof of the {C}ucker-{S}male flocking dynamics and
	mean-field limit.
	\newblock {\em Commun. Math. Sci.}, 7(2):297--325, 2009.
	
	\bibitem{HT2008}
	Seung-Yeal Ha and Eitan Tadmor.
	\newblock From particle to kinetic and hydrodynamic descriptions of flocking.
	\newblock {\em Kinet. Relat. Models}, 1(3):415--435, 2008.
	
	\bibitem{HeT2017}
	Siming He and Eitan Tadmor.
	\newblock Global regularity of two-dimensional flocking hydrodynamics.
	\newblock {\em Comptes rendus - Mathématique Ser. I}, 355:795–805, 2017.
	
	\bibitem{Isett2018}
	Philip Isett.
	\newblock A proof of {O}nsager's conjecture.
	\newblock {\em Ann. of Math. (2)}, 188(3):871--963, 2018.
	
	\bibitem{Darwin}
	 Laura Perea,  Gerard G\'omez and  Pedro Elosegui .
	\newblock Extension of the Cucker-Smale control law to space flight formations.
	\newblock {\em Journal of Guidance, Control and Dynamics,} 2009.
	
	\bibitem{LS-entropy}
	Trevor Leslie and Roman Shvydkoy.
	\newblock On the structure of limiting flocks in hydrodynamic Euler alignment
	models.
	\newblock To appear in {\em Math. Models Methods Appl. Sci.} 2019.
	
	\bibitem{MT2014}
	Sebastien Motsch and Eitan Tadmor.
	\newblock Heterophilious dynamics enhances consensus.
	\newblock {\em SIAM}, 56(4):577--621, 2014.
	
	\bibitem{Shv2018}
	Roman Shvydkoy.
	\newblock Global existence and stability of nearly aligned flocks.
	\newblock {\em J. Dynam. Differential Equations}, 31(4):2165--2175, 2018.
	
	\bibitem{ST-multi}
	Roman Shvydkoy and Eitan Tadmor.
	\newblock Multi-flocks: emergent dynamics in systems with multi-scale
	collective behavior.
	\newblock {\em Preprint.}
	
	\bibitem{ST-topo}
	Roman Shvydkoy and Eitan Tadmor.
	\newblock Topological models for emergent dynamics with short-range
	interactions.
	\newblock arXiv:1806.01371.
	
	\bibitem{ST1}
	Roman Shvydkoy and Eitan Tadmor.
	\newblock Eulerian dynamics with a commutator forcing I.
	\newblock {\em Trans. Math. Appl.},
	2017.
	
	\bibitem{ST2}
	Roman Shvydkoy and Eitan Tadmor.
	\newblock Eulerian dynamics with a commutator forcing {II}: {F}locking.
	\newblock {\em Discrete Contin. Dyn. Syst.}, 37(11):5503--5520, 2017.
	
	\bibitem{ST3}
	Roman Shvydkoy and Eitan Tadmor.
	\newblock Eulerian dynamics with a commutator forcing {III}. {F}ractional
	diffusion of order {$0<\alpha<1$}.
	\newblock {\em Phys. D}, 376/377:131--137, 2018.
	
	\bibitem{TT2014}
	Eitan Tadmor and Changhui Tan.
	\newblock Critical thresholds in flocking hydrodynamics with non-local
	alignment.
	\newblock {\em Philos. Trans. R. Soc. Lond. Ser. A Math. Phys. Eng. Sci.},
	372(2028), 2014.
	
	\bibitem{Tan2017}
	Changhui Tan.
\newblock {Singularity formation for a fluid mechanics model with nonlocal velocity. arXiv:1708.09360}
	\newblock  To appear in {\em Communications in Mathematical Sciences}, 2019.
	
	\bibitem{VZ2012}
	T.~Vicsek and A.~Zefeiris.
	\newblock Collective motion.
	\newblock {\em Physics Reprints}, 517:71--140, 2012.
	
\end{thebibliography}

\end{document}